\documentclass[a4paper,12pt]{amsart}

\usepackage[margin=1in]{geometry}

\usepackage{amsmath, amsthm, amssymb}
\usepackage{cite}
\usepackage[dvipdfmx]{graphicx}
\usepackage{bm}

\theoremstyle{definition}
\newtheorem{theorem}{Theorem}[section]
\newtheorem{lemma}[theorem]{Lemma}
\newtheorem{proposition}[theorem]{Proposition}
\newtheorem{fact}[theorem]{Fact} 
\newtheorem{observation}[theorem]{Observation} 

\theoremstyle{definition}
\newtheorem{definition}[theorem]{Definition}

\newtheorem{pclaim}[theorem]{Claim}
\newtheorem*{ac}{Acknowledgments} 

\theoremstyle{remark}
\newtheorem{remark}[theorem]{Remark}

\newenvironment{rmenum}{
\begin{enumerate}

}
{\end{enumerate}}

\newcommand{\parNei}[2]{N_{#1}(#2)}

\newcommand{\tcomp}[2]{\mathcal{G}(#1, #2)}

\newcommand{\distgt}[4]{\lambda(#3, #4; #1, #2)}
\newcommand{\distgtf}[5]{\lambda(#4, #5; #3; #1, #2)}
\newcommand{\parcut}[2]{\delta_{#1}(#2)}

\newcommand{\critical}[1]{\mathcal{F}_{\mathrm{cr}}(#1)}
\newcommand{\internal}[1]{int(#1)}

\newcommand{\guide}[3]{(#1, #2)_{#3}}
\newcommand{\eoo}{\unlhd_\circ}
\newcommand{\eo}{\unlhd}

\newcommand{\gtsim}[2]{\sim_{(#1, #2)}}
\newcommand{\gtpart}[2]{\mathcal{P}(#1, #2)}
\newcommand{\pargtpart}[3]{\mathcal{P}(#3; #1, #2)}

\newcommand{\vup}[1]{U(#1)}

\title[Bipartite Graft II]{Bipartite Graft II: Cathedral Decomposition for Combs}
\author{Nanao Kita}
\address{Tokyo University of Science 2641 Yamazaki, Noda, Chiba, Japan 278-0022}
\email{kita@rs.tus.ac.jp}

\date{\today}

\begin{document}

\begin{abstract} 
We provide a canonical decomposition for a class of bipartite grafts known as combs. 
As every bipartite graft is a recursive combination of combs, 
our results  provides a canonical decomposition for general bipartite grafts.   
Our new decomposition is by definition a generalization of the classical canonical decomposition in matching theory, 
that is, the Dulmage-Mendelsohn decomposition for bipartite graphs with perfect matchings. 
However, it exhibits much more complicated structure than its classical counterpart. 
It is revealed from our results that bipartite grafts has a canonical structure that is analogous to the cathedral decomposition 
for nonbipartite graphs with perfect matchings. 
\end{abstract}

\maketitle

\section{Definitions}
\subsection{Basic Notation} 
We mostly follow Schrijver~\cite{schrijver2003} for standard notation and definitions.  
We list in the following exceptions or nonstandard definitions.  
For two sets $A$ and $B$, we denote the symmetric difference $(A\setminus B) \cup (B\setminus A)$ by $A\Delta B$. 
We often denote a signleton $\{x\}$ simply by $x$. 
For a graph $G$, we denote the vertex and edge sets of $G$ by $V(G)$ and $E(G)$, respectively.  
We treat multigraphs. 
For two vertices $u$ and $v$ from a graph, $uv$ denotes an edge between $u$ and $v$. 
We consider paths and circuits as graphs. 
That is, a circuit is a connected graph in which every vertex is of degree two, 
whereas a path is a connected graph in which every vertex is of degree two or less 
and at least one vertex is of degree less than two. 
A graph with a single vertex and no edge is a path.   
For a path $P$ and vertices $x, y \in V(P)$, 
we denote by $xPy$ the subpath of $P$ whose ends are $x$ and $y$. 
We often treat a graph $G$ as if it is the set $V(G)$ of vertices. 
A perfect matching of a graph $G$ is a set $M$ of edges such that 
every vertex of $G$ is connected to exactly one edge from $M$.

In the remainder of this section, let $G$ be a graph unless stated otherwise. 
Let $X\subseteq V(G)$. 
The subgraph of $G$ induced by $X$ is denoted by $G[X]$. 
The graph $G[V(G)\setminus X]$ is denoted by $G- X$. 
The graph obtained from contracting  $X$ is denoted by $G/X$.

For $X, Y \subseteq V(G)$, 
we denote by $E_G[X, Y]$ the set of edges between $X$ and $Y$. 
The set $E_G[X, V(G)\setminus X]$ is denoted by $\parcut{G}{X}$. 
The set of edges that joins vertices in $X$ is denoted by $E[X]$. 
The set of neighbors of $X$ is denoted by $\parNei{G}{X}$. 

Let $G_1$ and $G_2$ be subgraphs of $G$. 
The addition of $G_1$ and $G_2$ is denoted by $G_1 + G_2$. 
Let $F\subseteq E(G)$.  For a subgraph $H$ of $G$, 
$H + F$ denotes the graph obtained by adding $F$ to $H$. 
The subgraph of $G$ determined by $F$ is denoted by $G. F$. That is, $G. F = (V(G), F)$.

\subsection{Posets}

Let $(X, \le)$ be a poset with the minimum element $r \in X$. 
Let $x\in X$, and let $x_1, \ldots, x_k \in X$, where $k\ge 1$, be a sequence with the minimum number of mutually distinct elements such that 
$x_1 = r$, $x_k = x$, and 
$x_i \le x_{i+1}$ holds for each $i \in \{1, \ldots, k\} \setminus \{k\}$, 
whereas $x_i \le y \le x_{i+1}$ implies  $x_i = y$ or $x_{i+1} = y$ for every $y\in X$.  
We call the number $k$ the {\em height} of $x$ in the poset $(X, \le)$.

\subsection{Ears} 

Let $G$ be a graph, and let $X\subseteq V(G)$. 
Let $P$ be a subgraph of $G$ with $E(P)\neq \emptyset$. 
We call  $P$  a {\em round ear} relative to $X$ 
if $P$ is a path  whose vertices except for its ends are disjoint from $X$ 
or a circuit that shares exactly one vertex with $X$. 
We call  $P$ a {\em straight ear} relative to $X$ 
if $P$ is a path one of whose ends is the only vertex that $P$ shares with $X$. 
We call $P$ an {\em ear} if $P$ is either a round or straight ear. 
For an ear $P$ relative to $X$,   
vertices in $V(P)\cap X$ are called  {\em bonds } of $P$. 
Vertices of $P$ except for bonds are  called  {\em internal} vertices of $P$ even if $P$ is a straight ear. 
The set of internal vertices of $P$ is denoted by $\internal{P}$. 
An edge of $P$ between $X$ and $\internal{P}$ is called a {\em neck} of $P$.

\subsection{Ear Decompositions}

\begin{definition} 
Let $G$ be a graph, and let $r\in V(G)$. 
Let $ G_1, \ldots, G_l $, where $l \ge 1$,  be a sequence of subgraphs of $G$ such that 
\begin{rmenum} 
\item $G_1 = ( \{r\}, \emptyset)$, 
\item for each $i\in \{1,\ldots, l\}\setminus \{l\}$, $G_{i+1} = G_i + P_i$, where $P_i$ is a subgraph of $G$ that is an ear relative to $G_i$, and 
\item $G_l = G$. 
\end{rmenum} 
The set $\{P_i : i\in \{1,\ldots, l\}\setminus \{l\}\}$ is called an {\em ear decomposition} of $G$. 
\end{definition} 

\begin{definition} 
Let $G$ be a graph, and let $\mathcal{P}$ be an ear decomposition of $G$. 
Define a binary relation $\eoo$ as follows: 
For each $P, P' \in \mathcal{P}$, let $P \eoo P'$ if $P = P'$ or a bond of $P'$ is contained in $\internal{P}$. 
Furthermore, define  $\eo$ as follows: 
For each $P, P' \in \mathcal{P}$, let $P \eo P'$ if there exist $Q_1, \ldots, Q_k \in \mathcal{P}$, where $1 \le k \le |\mathcal{P}|$, such that 
$P = Q_1$, $P' = Q_k$, and $Q_i \le Q_{i+1}$ for each $i \in \{1,\ldots, k\} \setminus \{k\}$.  
\end{definition} 

The binary relation $\eo$ is clearly a partial order.

\section{Grafts and Joins} 
\subsection{Basic Definitions} 
Let $(G, T)$ be a pair of a graph $G$ and a set $T\subseteq V(G)$. 
A set $F\subseteq E(G)$ is a {\em join} of $(G, T)$ 
if $|\parcut{G}{v}\cap F|$ is odd for each $v\in T$ and is even for each $v\in V(G)\setminus T$. 
The pair $(G, T)$ is a {\em graft} if  $|V(C)\cap T|$ is even  for every connected component  $C$ of $G$. 

\begin{fact} \label{fact:graft} 
Let  $G$ be a graph,  and let $T\subseteq V(G)$. 
Then, $(G, T)$ has a join if and only if it is a graft. 
\end{fact} 

Joins in a graft is also referred to as $T$-joins in a graph. 
Under Fact~\ref{fact:graft}, we are typically interested in {\em minimum joins}, 
that is, joins with the minimum number of edges. 
Minimum joins in a graft is a generalization of perfect matchings in a graph with perfect matchings~\cite{lp1986}. 
We denote  by $\nu(G, T)$ the number of edges in a minimum join of a graft $(G, T)$. 

We often treat items of $G$ as items of $(G, T)$. 
For example, we call a path in $G$ a path in $(G, T)$. 
A graft $(G, T)$ is {\em bipartite} if $G$ is bipartite. 
We call sets $A$ and $B$ color classes of $(G, T)$ if these are color classes of $G$.  
If $(H, T\cap V(H))$ is a graft for a subgraph $H$ of $G$, 
we say that $(H, T\cap V(H))$ is a {\em subgraft} of $(G, T)$.

\subsection{Factor-Components} 

Let $(G, T)$ be a graft. 
An edge $e \in E(G)$ is {\em allowed} if $(G, T)$ has a minimum join that contains $e$. 
Vertices $u,v\in V(G)$ are {\em factor-connected} if 
there is a path between $u$ and $v$ whose edges are allowed in $(G, T)$. 
A graft is {\em factor-connected} if every two vertices are factor-connected. 
A maximal factor-connected subgraft is called a {\em factor-connected component} or {\em factor-component}. 
The set of factor-components of a graft $(G, T)$ is denoted by $\tcomp{G}{T}$.  

We say that a set $X\subseteq V(G)$ is {\em separating} 
if there exists $\mathcal{H}  \subseteq \tcomp{G}{T}$ such that $X = \bigcup_{H\in\mathcal{H}} V(H)$. 

Note that a set $X \subseteq V(G)$ is separating if and only if 
$\parcut{G}{X}\cap F = \emptyset$ for every minimum join $F$ of $(G, T)$.  

\subsection{Operations on Grafts} 

Let $(G, T)$ be a graft, and let $X\subseteq V(G)$. 
If $X$ is separating, then $(G[X], T\cap X)$ is a graft and is denoted by $(G, T)[X]$. 

\begin{observation} 
For a separating set $X\subseteq V(G)$, 
a set $F \subseteq E(G)$ is a minimum join of $(G, T)[X]$ if and only if 
there is a minimum join $\hat{F}$ of $(G, T)$ that contains $F$. 
\end{observation}

We define the {\em contraction} of $(G, T)$ by $X$ 
as the graft $(G/X, T\setminus X)$  or $(G/X, (T\setminus X) \cup \{ x\} )$, where $x$ denotes the contracted vertex corresponding to $X$,
  if $|X\cap T|$ is even or odd, respectively. 
The contraction of $(G, T)$ by $X$ is denoted by $(G, T)/X$. 

\begin{observation} 
If $F$ is a join of a graft $(G, T)$, 
then $F\setminus E[X]$ is also a join of the graft $(G, T)/X$ for every $X\subseteq V(G)$. 
\end{observation} 

Note that $F\setminus E[X]$  may not be a minimum join of $(G, T)/X$ even if $F$ is a minimum join of $(G, T)$.

Let $(G_1, T_1)$ and $(G_2, T_2)$ be grafts. 
We define the addition of $(G_1, T_1)$ and $(G_2, T_2)$ as 
the graft $(G_1 + G_2, T_1\Delta T_2)$ and denote this by $(G_1, T_1)\oplus(G_2, T_2)$.

\begin{observation} \label{obs:addition}
If $F_1$ and $F_2$ are joins of grafts $(G_1, T_1)$ and $(G_2, T_2)$, respectively, 
then $F_1 \cup F_2$ is a join of the graft $(G_1, T_1)\oplus(G_2, T_2)$. 
\end{observation}

Note that $F_1 \cup F_2$ may not be a minimum join even if $F_1$ and $F_2$ are minimum joins of $(G_1, T_1)$ and $(G_1, T_2)$.

\subsection{Bipartite Grafts with Ordered Bipartition}

We call $(G, T; A, B)$ a {\em bipartite graft with an ordered bipartition}  or sometimes simply a (bipartite) graft 
if $(G, T)$ is a bipartite graft with color classes $A$ and $B$. 
Note that $(G, T; A, B)$ and $(G, T; B, A)$ 
are not equivalent as bipartite grafts with an ordered bipartition. 
We often call a bipartite graft with an ordered bipartition simply as a bipartite graft or a graft. 
That is, 
if we say that $(G, T; A, B)$ a (bipartite) graft, it implies that $(G, T; A, B)$ is a bipartite graft with an ordered bipartition.

\begin{definition} 
Let $(G_1, T_1; A_1, B_1)$ and $(G_2, T_2; A_2, B_2)$ be bipartite grafts with $A_1 \cap B_2 = \emptyset$ and $A_2 \cap B_1 = \emptyset$.  
We define the addition of $(G_1, T_1; A_1, B_1)$ and $(G_2, T_2; A_2, B_2)$ as 
the bipartite graft $(G_1 + G_2, T_1\Delta T_2; A_1 \cup A_2, B_1 \cup B_2)$ 
and denote this by $(G_1, T_1; A_1, B_1)\oplus(G_2, T_2; A_2, B_2)$. 
\end{definition}

\section{Distances in Grafts} 

In this section, we explain the concept of distances between vertices in a graft 
that are determined by a given minimum join, and present two fundamental properties. 

\begin{definition} 
Let $(G, T)$ be a graft. Let $F \subseteq E(G)$. 
We define $w_F: E(G)\rightarrow \{1, -1\}$ as  
$w_F(e) = 1$ for $e\in E(G)\setminus F$ and $w_F(e) = -1$ for $e\in F$. 
For a subgraph $P$ of $G$, which is typically a path or circuit, 
 $w_F(P)$ denotes $ \Sigma_{ e \in E(P) } w_F(e)$, 
 and is referred to as the $F$-{\em weight} of $P$. 

For $u,v\in V(G)$, 
a path between $u$ and $v$ with the minimum $F$-weight is said to be {\em $F$-shortest} between $u$ and $v$.  
The $F$-weight of an $F$-shortest path between $u$ and $v$ is referred to as 
the $F$-{\em distance} between $u$ and $v$, 
and is denoted by $\distgtf{G}{T}{F}{u}{v}$. 
Regarding these terms, we sometimes omit the prefix ``$F$-'' if the meaning is apparent from the context. 
\end{definition}

The next proposition implies that $F$-distances do not depend on the choice of the minimum join $F$. 

\begin{proposition}[Seb\"o~\cite{DBLP:journals/jct/Sebo90}] \label{prop:dist2invariant} 
Let $(G, T)$ be a graft, and let $F$ be a minimum join. 
Then, $\distgtf{G}{T}{F}{u}{v} = \nu(G, T) - \nu(G, T\Delta \{u, v\})$ for every $u, v\in V(G)$. 
\end{proposition} 

Under Proposition~\ref{prop:dist2invariant}, 
we sometimes denote $\distgtf{G}{T}{F}{u}{v}$ simply as $\distgt{G}{T}{u}{v}$.

\begin{lemma}[see Kita~\cite{kita2020bipartite}] \label{lem:elem2nonposi} 
If a graft $(G, T)$ is factor-connected, then $\distgt{G}{T}{x}{y} \le 0$ for every $x, y\in V(G)$. 
\end{lemma}

\section{Fundamental Properties of Minimum Joins} 

We present two fundamental properties of minimum joins. 
We use the following two lemmas everywhere in this paper. 
These lemmas can easily be derived by considering the symmetric differences of joins. 
See also Seb\"o~\cite{DBLP:journals/jct/Sebo90}. 

\begin{lemma} \label{lem:minimumjoin} 
Let $(G, T)$ be a graft, and let $F$ be a join of $(G, T)$. 
Then, $F$ is a minimum join of $(G, T)$ if and only if 
there is no circuit $C$ with $w_F(C) < 0$. 
\end{lemma}

\begin{lemma} \label{lem:circuit}  
Let $(G, T)$ be graft, and let $F$ be a minimum join. 
If $C$ is a circuit with $w_F(C) = 0$, then $F\Delta E(C)$ is also a minimum join of $(G, T)$. 
Accordingly, every edge of $C$ is allowed. 
\end{lemma}

\section{General Kotzig-Lov\'asz Decomposition} 

In this section, we introduce a canonical decomposition for grafts known as 
the general Kotzig-Lov\'asz decomposition. 
We use this decomposition in Section~\ref{sec:relationship} 
when we discuss its relationship with our new decomposition.

\begin{definition} 
Let $(G, T)$ be a graft. 
Define a binary relation as follows: For two vertices $u, v\in V(G)$,   
let $u \gtsim{G}{T} v$ if $u$ and $v$ are identical or if they are contained in the same factor-component
and $\nu(G, T\Delta \{u, v\}) = \nu(G, T)$.  
\end{definition}

\begin{theorem}[Kita~\cite{kita2017parity}] \label{thm:tkl} 
For every graft $(G, T$), the binary relation $\gtsim{G}{T}$ is an equivalence relation. 
\end{theorem} 

Under Theorem~\ref{thm:tkl}, 
the family of equivalence classes of $\gtsim{G}{T}$ is  called the {\em general Kotzig-Lov\'asz decomposition} 
or the {\em Kotzig-Lov\'asz decomposition} of the graft $(G, T)$, 
and is denoted by $\gtpart{G}{T}$. 
It is obvious from the definition that, for each $H \in \tcomp{G}{T}$, 
the family $\{ S \in\gtpart{G}{T}: S\cap V(H) \neq \emptyset\}$ is a partition of $V(H)$. 
We denote this family by $\pargtpart{G}{T}{H}$. 

Note the next lemma, which can be confirmed from Proposition~\ref{prop:dist2invariant} and Lemma~\ref{lem:elem2nonposi}.

\begin{lemma} \label{lem:tkl2dist} 
Let $(G, T)$ be a graft and  $F$ be a minimum join. 
Let $H\in\tcomp{G}{T}$ and $u,v \in V(H)$. 
Then, $u \gtsim{G}{T} v$ holds if and only if $\distgtf{G}{T}{F}{u}{v} = 0$, 
whereas $u \gtsim{G}{T} v$ does not hold if and only if $\distgtf{G}{T}{F}{u}{v} < 0$. 
\end{lemma}

\section{Combs}

\begin{definition} 
Let $(G, T)$ be a bipartite graft with color classes $A$ and $B$. 
If $B\subseteq T$ and $\nu(G, T) = |B|$, then  $(G, T)$ is called a {\em quasicomb}, 
and sets $A$ and $B$ are called the {\em spine} and {\em tooth} sets of $(G, T)$, respectively. 
If we say that $(G, T; A, B)$ is a comb, then $A$ and $B$ are the spine and tooth sets, respectively.     
\end{definition}

\begin{remark} 
Note that the definition of combs can be rephrased as follows. 
Let $(G, T; A, B)$ be a bipartite graft, and let $F$ be a minimum join. 
Then, $(G, T; A, B)$ is a comb if and only if $|\parcut{G}{v} \cap F | = 1$ for every $v\in B$. 
\end{remark}

The next lemma can easily be confirmed from Lemma~\ref{lem:elem2nonposi} and the definition of combs.

\begin{lemma}[Kita~\cite{kita2020bipartite}] \label{lem:elemcomb2nonposi} 
If $(G, T; A, B)$ is a factor-connected comb, then 
\begin{rmenum} 
\item $\distgt{G}{T}{x}{y} = 0$ for every $x, y\in A$,  
\item $\distgt{G}{T}{x}{y} = -1$ for every $x\in A$ and every $y\in B$, and 
\item $\distgt{G}{T}{x}{y} \in \{ 0, -2\}$ for every $x, y \in B$. 
\end{rmenum} 
\end{lemma} 

The next lemma can easily be derived from Lemmas~\ref{lem:tkl2dist} and \ref{lem:elemcomb2nonposi}. 
In Section~\ref{sec:relationship}, new results are stated on the premise of  this lemma. 

\begin{lemma}[Kita~\cite{kita2020bipartite}] \label{lem:comb2kl} 
Let $(G, T; A, B)$ be a comb, and let $H\in\tcomp{G}{T}$. 
Then, $V(H)\cap A \in \pargtpart{G}{T}{H}$ holds, 
and the family $\pargtpart{G}{T}{H}\setminus \{ V(H)\cap A\}$ is a partition of $V(H)\cap B$. 
\end{lemma}

\section{Quasicombs}

\begin{definition} 
Let $(G, T)$ be a bipartite graft with color classes $A$ and $B$. 
If $\nu(G, T) = |B\cap T|$, then the graft $(G, T)$ is called a {\em quasicomb}, 
and sets $A$ and $B$ are called the {\em spine} and {\em tooth} sets of $(G, T)$, respectively. 
If we say that $(G, T; A, B)$ is a quasicomb, then $A$ and $B$ are the spine and tooth sets, respectively.     
\end{definition}

\begin{remark} 
If $(G, T; A, B)$ is a comb, then it is a quasicomb. 
\end{remark}

\begin{remark} 
The definition of quasicombs can be rephrased as follows. 
Let $(G, T; A, B)$ be a bipartite graft, and let $F$ be a minimum join. 
Then, $(G, T; A, B)$ is a quasicomb if and only if $|\parcut{G}{v} \cap F | \le 1$ for every $v\in B$. 
\end{remark}

In the following, we present four fundamental lemmas on quasicombs that can easily be confirmed.  
We use these lemmas everywhere in this paper sometimes without explicitly mentioning them. 
Note that these lemmas can be applied to combs because combs are quasicombs.

\begin{lemma} \label{lem:qcomb2dist} 
Let $(G, T; A, B)$ be a quasicomb, and let $F$ be a minimum join of $(G, T; A, B)$. 
Then, the following properties hold. 
\begin{rmenum} 
\item $\distgtf{G}{T}{F}{x}{y} \ge 0$ for every $x, y \in A$. 
\item $\distgtf{G}{T}{F}{x}{y} \ge -1$ for every $x\in A$ and every $y\in B$. 
\item $\distgtf{G}{T}{F}{x}{y} \ge -2$ for every $x, y\in B$.  
\end{rmenum} 
\end{lemma}

\begin{definition}  
Let $(G, T; A, B)$ be a graft, and let $F \subseteq E(G)$.  
We say that a path $P$ between $s$ and $t$ is {\em $F$-balanced} 
if $| \parcut{P}{v} \cap F | = 1$ holds for every $v\in ( V(P)\cap B ) \setminus \{s, t\}$.   
A circuit $C$ is {\em $F$-balanced } 
if $| \parcut{P}{v} \cap F | = 1$ holds for every $v\in V(C)\cap B$. 
An ear $P$ relative to a set of vertices is {\em F-balanced} 
if $|\parcut{P}{v} \cap F| = 1$  for every $v\in \internal{P} \cap B $.  
\end{definition}

\begin{lemma} \label{lem:balancedcircuit} 
Let $(G, T; A, B)$ be a quasicomb, and let $F$ be a minimum join of $(G, T; A, B)$.  
A circuit $C$ is $F$-balanced if and only if $w_F(C) = 0$. 
\end{lemma}

\begin{lemma} \label{lem:qcomb2path} 
Let $(G, T; A, B)$ be a quasicomb, and let $F$ be a minimum join of $(G, T; A, B)$.  
Let $x, y\in V(G)$, and let $P$ be a path between $x$ and $y$.   
\begin{rmenum} 
\item Let $x, y\in A$. 
Then, $P$ is $F$-balanced if and only if $w_F(P ) = 0$. 
\item Let $x\in A$ and $y\in B$.  
\begin{enumerate} 
\item 
If $P$ is $F$-balanced then $w_F(P) \in \{ 1, -1 \}$ holds. 
\item  $w_F(P ) = -1$ holds if and only if $P$ is $F$-balanced and the edge of $P$ connected to $y$ is in $F$. 
\item Assume that $P$ is $F$-balanced. 
Then, $w_F(P) = 1$ holds if and only if the edge of $P$ connected to $y$ is not in $F$. 
\end{enumerate} 
\item Let $x, y\in B$. 
\begin{enumerate} 
\item If $P$ is $F$-balanced then $w_F(P)  \in \{ -2, 0, 2\}$ holds. 
\item $w_F(P) = -2$ holds if and only if $P$ is $F$-balanced and the edges of $P$ connected to the ends are in $F$. 
\item Assume that $P$ is $F$-balanced. 
Then, $w_F(P) = 0$ holds if and only if, of the two edges of $P$ connected to the ends,  one is in $F$ whereas the other is not in $F$. 
Additionally, $w_F(P)  =2$ holds if and only if the edges of $P$ connected to the ends are not in $F$. 
\end{enumerate} 
\end{rmenum} 
\end{lemma}

\begin{lemma} \label{lem:nonbalanced}  
Let $(G, T; A, B)$ be a quasicomb, and let $F$ be a minimum join. 
Let $x, y\in B$, and let $P$ be a path between $x$ and $y$ with $w_F(P) = 0$. 
Then, there uniquely exists a vertex $z \in V(P)\cap B$ with $\parcut{P}{z} \cap F = \emptyset$, 
and  $xPz$ and $yPz$ are $F$-balanced paths of weight zero. 
If $P$ is $F$-balanced, then either $z = x$ or $z = y$ holds. 
If $P$ is not $F$-balanced,  
then, in $xPz$ and $yPz$, the edges connected to the end $x$ or $y$ are from $F$, 
whereas the edges connected to $z$ are from $E(G)\setminus F$.

\end{lemma}

Under Lemma~\ref{lem:nonbalanced}, we define the concept of midvertices. 
\begin{definition} 
Let $(G, T; A, B)$ be a quasicomb, and let $F$ be a minimum join. 
Let $x, y\in B$, and let $P$ be a path between $x$ and $y$ whose $F$-weight is zero. 
We call the vertex $z\in V(P)\cap B$ with $\parcut{P}{z}\cap F = \emptyset$ 
the {\em midvertex} of $P$.

\end{definition}

\section{Critical Quasicombs} 

\begin{definition}
Let $(G, T; A, B)$ be a quasicomb. 
Let $r\in B$. 
We say that $(G, T; A, B)$ is {\em critical } with {\em root } $r$ if 
$\nu(G,T) - \nu(G, T\Delta \{x, r\}) = 1$ for each $x \in A$ and $\nu(G,T) - \nu(G, T\Delta \{x, r\}) = 0$ for each $x\in B$. 
\end{definition} 

We use the following two lemmas everywhere in this paper sometimes without explicitly mentioning them. 
Under Proposition~\ref{prop:dist2invariant}, 
the definition of critical quasicombs can be rephrased as follows. 

\begin{lemma} \label{lem:critical2dist} 
Let $(G, T; A, B)$ be a quasicomb, and let $r\in B$. 
Let $F$ be a minimum join. 
Then, the following properties are equivalent. 
\begin{rmenum} 
\item $(G, T; A, B)$ is a critical quasicomb with root $r$. 
\item $\distgtf{G}{T}{F}{x}{r} = 1$ for each $x \in A$, and $\distgtf{G}{T}{F}{x}{r} = 0$ for each $x\in B$. 
\end{rmenum} 
\end{lemma}

Lemmas~\ref{lem:qcomb2path} and \ref{lem:critical2dist} easily imply the next lemma.

\begin{lemma}  \label{lem:critical} 
Let $(G, T; A, B)$ be a critical quasicomb whose root is $r\in B$. 
Let $F$ be a minimum join of $(G, T)$. 
Then, the following properties hold. 
\begin{rmenum} 
\item $\parcut{G}{r} \cap F = \emptyset$. 
\item $|\parcut{G}{v} \cap F| = 1$ for each $v\in B\setminus \{r\}$. 
\item Every $F$-shortest path between $v\in B \setminus \{r\}$ and $r$
is an $F$-balanced path of weight $0$ in which $v$ is connected to an edge from $F$.  
\item Every $F$-shortest path between $v\in A$ and $r$ is an $F$-balanced path of weight $1$.  
\end{rmenum} 
\end{lemma}

\section{Ear Grafts}

\begin{definition} 
Let $(G, T; A, B)$ and $(P, T'; A', B')$ be bipartite grafts with an ordered bipartition. 
We say that $(P, T'; A', B')$ is an {\em ear graft}  relative to $(G, T; A, B)$ if the following are satisfied. 
\begin{rmenum} 
\item $A \cap B' = \emptyset$ and $A' \cap B = \emptyset$. 
\item $P$ is an ear relative to $G$; let $s$ and $t$ be its bonds. 
\item $( V(P)\setminus \{s, t\}) \cap B \subseteq T$ holds, and 
\item $\{s, t\} \cap B \cap T = \emptyset$. 
\end{rmenum} 
If $P$ is a round or straight ear,  then $(P, T'; A', B')$ is said to be a {\em round} or {\em straight} ear graft, respectively. 
\end{definition} 

Note that if 
 $(P, T'; A', B')$ is an ear graft relative to $(G, T; A, B)$ with bonds $s$ and $t$ then 
 $(P, T'; A', B')$ is a quasicomb with $B'\cap T' = B' \setminus \{s, t\}$.

\begin{definition} 
Let $(G, T; A, B)$  be a graft, and let $(P, T'; A', B')$ be an ear graft relative to $(G, T; A, B)$. 
We say that $(P, T'; A', B')$ is {\em effective} if 
it is round or if  it is straight and satisfies the following \ref{item:freeend} and \ref{item:bond}: 
\begin{rmenum} 
\item \label{item:freeend} If $s\in A$ holds, then $s\not\in T$ holds; and 
\item \label{item:bond} if $t\in A$ holds, then $t \in T$ holds, 
\end{rmenum} 
where $s$ and $t$ are the ends of $P$ between which $t$ is the bond. 
\end{definition}

The next lemma can easily be confirmed. 

\begin{lemma} \label{lem:lear2matching} 
Let $(G, T; A, B)$  be a graft, and 
let $(P, T'; A', B')$ be an ear graft relative to $(G, T; A, B)$  that is effective and straight. 
Let $s$ and $t$ be the ends of $P$ between which $t$ is the bond. 
Then,  $V(P)\setminus \{s, t\} \subseteq T$ holds. 
Additionally, the only mininmum join of $(P, T'; A, B)$ is the perfect matching of $P - \{s\}\cap A - \{t\}\cap B$.  
\end{lemma}

Then next lemma is easily confirmed from the definition of effective ear grafts and Lemma~\ref{lem:lear2matching}.

\begin{lemma} \label{lem:ear2path} 
Let $(G, T; A, B)$  be a graft, and 
let $(P, T'; A', B')$ be an effective ear graft relative to $(G, T; A, B)$ with bonds $s$ and $t$. 
Let  $F$ be a minimum join of $(P, T'; A', B')$. 
Then, the following properties hold. 
\begin{rmenum} 
\item For each $x\in V(P)\cap A'$, there exists $r \in \{s, t\}$ such that 
either $r\in A'$ holds and $xPr$ is an $F$-balanced path of weight $0$ 
or $r\in B'$ holds and $xPr$ is an $F$-balanced path of weight $1$. 
\item For each $x\in V(P)\cap B'$,   
 there exists $r \in \{s, t\}$ such that 
either $r\in A'$ holds and $xPr$ is an $F$-balanced path of weight $-1$ 
or $r\in B'$ holds and $xPr$ is an $F$-balanced path of weight $0$ in which $x$ is connected to an edge from $F$.  
\end{rmenum} 
\end{lemma}

\section{Constructive Characterization for Critical Quasicombs}  

From this section onward, we present our main results. 
In this section, we provide a constructive characterization of critical quasicombs.  
The main results in this section, Theorems~\ref{thm:critical2char}  and \ref{thm:ear2balanced},  
are used in Section~\ref{sec:cath} when we derive our new decomposition. 

\begin{definition} 
We define a family $\critical{r}$ of bipartite grafts with an ordered bipartition that have a vertex $r$ as follows. 
\begin{rmenum} 
\item The graft $( (\{r\}, \emptyset), \emptyset; \emptyset, \{r\})$ is a member of $\critical{r}$. 
\item Let $(G, T; A, B) \in \critical{r}$, and let $(P, T'; A', B')$ be an effective ear graft relative to $(G, T; A, B)$. 
Then, $(G, T; A, B) \oplus (P, T'; A', B')$  is a member of $\critical{r}$. 
\end{rmenum} 
\end{definition} 

In the following,  we provide and prove Lemmas~\ref{lem:ear2critical}, \ref{lem:critical2increment}, and \ref{lem:critical2ear}  
 to prove in Theorem~\ref{thm:critical2char} that $\critical{r}$ is the family of critical quasicombs with root $r$.

\begin{lemma} \label{lem:ear2critical} 
Every member of $\critical{r}$ is a critical quasicomb with root $r$. 
\end{lemma} 
\begin{proof} 
We prove this lemma by induction along the definition of $\critical{r}$. 
First, the statement clearly holds for $( (\{r\}, \emptyset), \emptyset; \emptyset, \{r\} )$. 
Next, let $(G, T; A, B) \in \critical{r}$, 
and assume that the statement holds for $(G, T; A, B)$; that is, $(G, T; A, B)$ is a critical quasicomb with root $r$. 
Let $(P, T'; A', B')$ be an ear graft relative to $(G, T)$. 
Let $F$ and $F'$ be minimum joins of $(G, T)$ and $(P, T')$, respectively. 
Let $(\hat{G}, \hat{T}; \hat{A}, \hat{B}):= (G, T; A, B) \oplus (P, T'; A', B')$. 
Let $\hat{F} := F \cup F'$. 
As is mentioned in Observation~\ref{obs:addition}, $\hat{F}$ is a join of $(\hat{G}, \hat{T}; \hat{A}, \hat{B})$. 

\begin{pclaim} \label{claim:join} 
The graft $(\hat{G}, \hat{T}; \hat{A}, \hat{B})$ is a quasicomb, 
and $\hat{F}$ is a minimum join.  
\end{pclaim} 
\begin{proof} 
We first prove that $\hat{F}$ is a minimum join. 
Under Lemma~\ref{lem:minimumjoin}, we only need to prove that there is no circuit with negative $\hat{F}$-weight. 
If $C$ is a circuit in $(\hat{G}, \hat{T}; \hat{A}, \hat{B})$ with negative $\hat{F}$-weight, 
then $P$ is required to be a round ear with distinct bonds $s$ and $t$, 
and $C$ is the addition of $P$ and a path $Q$ in $(G, T; A, B)$ between $s$ and $t$. 
Hence, the claim can be proved if we prove $w_{F'}(P) + \distgtf{G}{T}{F}{s}{t} \ge 0$ 
under the supposition that  $P$ is a round ear with distinct bonds $s$ and $t$. 

First, consider the case with $s, t\in B$. 
Lemma~\ref{lem:qcomb2path} implies  $w_{\hat{F}}(P) = 2$, whereas Lemma~\ref{lem:qcomb2dist} implies $\distgtf{G}{T}{F}{s}{t} \ge -2$.  
Hence, $w_{F'}(P) + \distgtf{G}{T}{F}{s}{t} \ge 0$.   
For the other cases, where the bonds $s$ and $t$ are both in $A$ or individually in $A$ and $B$,  
similar discussions prove $w_{F'}(P) + \distgtf{G}{T}{F}{s}{t} \ge 0$.   
Therefore, no circuit in $(\hat{G}, \hat{T}; \hat{A}, \hat{B})$ can be of negative $\hat{F}$-weight; 
accordingly, $\hat{F}$ is a minimum join. 
This further implies that $\nu(\hat{G}, \hat{T}) = |\hat{B} \setminus \{r\}| = |\hat{B} \cap \hat{T}|$. 
That is, $(\hat{G}, \hat{T}; \hat{A}, \hat{B})$ is a quasicomb. 
Thus, the claim is proved. 
\end{proof}

Under Claim~\ref{claim:join}, we further prove that  $(\hat{G}, \hat{T}; \hat{A}, \hat{B})$ is a critical quasicomb. 

\begin{pclaim} \label{claim:critical} 
$\distgtf{\hat{G}}{\hat{T}}{\hat{F}}{r}{x} = 1$  holds for every $x\in \hat{A}$, whereas   
$\distgtf{\hat{G}}{\hat{T}}{\hat{F}}{r}{x} = 0$  holds for every $x\in \hat{B}$. 

\end{pclaim} 
\begin{proof} 
Under the induction hypothesis and Lemma~\ref{lem:critical}, the claim obviously holds for every $x \in V(G)$. 
This further proves the claim for every $x\in V(P)\setminus V(G)$ from Lemma~\ref{lem:ear2path}.

\end{proof} 

Under Claims~\ref{claim:join} and \ref{claim:critical}, 
Lemma~\ref{lem:critical} implies that $(\hat{G}, \hat{T}; \hat{A}, \hat{B})$ is a critical quasicomb with root $r$. 
This completes the proof of the lemma. 
\end{proof}

The next lemma is provided for proving Lemma~\ref{lem:critical2ear}. 

\begin{lemma} \label{lem:critical2increment} 
Let $(G, T; A, B)$ be a critical quasicomb with root $r\in B$, and let $F$ be a minimum join of $(G, T; A, B)$.  
Let $(G', T'; A', B')$ be a subgraft of $(G, T; A, B)$, where $A'\subseteq A$ and $B'\subseteq B$, 
such that $r\in B'$, $V(G') \subsetneq V(G)$, and $E_G[B', A\setminus A']\cap F = \emptyset$ hold.  
Then, $(G, T; A, B)$ has an ear graft $(G'', T''; A'', B'')$ relative to $(G', T'; A', B')$ 
with $V(G'')\setminus V(G') \neq \emptyset$ and $E_G[B'', A\setminus (A' \cup A'') ]\cap F = \emptyset$.  
\end{lemma} 
\begin{proof} 
First, we prove the case where $E_G[B', A\setminus A'] \neq \emptyset$. 
Let $e\in E_G[B', A\setminus A']$, and let $x\in A\setminus A'$ and $y\in B'$ be the ends of $e$. The assumption implies $e\not\in F$. 
Hence, $(G.e, \emptyset; \{x\}, \{y\})$ is an ear graft relative to $(G', T'; A', B')$ that meets the condition. 

Next, we prove the case where $E_G[B', A\setminus A'] = \emptyset$.  
Because $G$ is obviously connected, we have $E_G[A', B\setminus B'] \neq \emptyset$. 
Let $f\in E_G[A', B\setminus B']$, and let $u\in A'$ and $v\in B\setminus B'$ be the ends of $f$. 

First, consider the case with $f\not\in F$. 
Under Lemma~\ref{lem:critical2dist}, let $P$ be a path between $v$ and $r$ with $w_F(P) = 0$. 
Lemma~\ref{lem:critical} implies that $P$ is $F$-balanced, and the edge of $P$ connected to $v$ is in $F$. 
Hence, $e \not\in E(P)$ follows. 
Trace $P$ from $v$, and let $w$ be the first encountered vertex in $V(G')$. The assumption implies $w\in A'$. 
Then, $( vPw + f, V(vPw) \setminus \{v, w \}; A\cap V(vPw + f), B\cap V(vPw + f) )$ is a desired ear graft relative to $(G', T'; A', B')$ that meets the condition.

For the remaining case with $f\in F$, $(G. f, \{u, v\}; \{u\}, \{v\})$ is a desired ear graft relative to $(G', T'; A', B')$.  
This completes the proof. 
\end{proof}

Lemma~\ref{lem:critical2increment} derives the next lemma.

\begin{lemma} \label{lem:critical2ear} 
Every critical quasicomb with root $r$ is a member of $\critical{r}$. 
\end{lemma} 
\begin{proof} 
Let $(G, T; A, B)$ be a critical quasicomb with root $r$. 
We prove the lemma by induction on $V(G)$. 
If $| V(G) | = 1$, that is, $(G, T; A, B)$ is $( (\{r\}, \emptyset), \emptyset; \emptyset, \{r\})$, 
then  $(G, T; A, B)\in \critical{r}$ obviously holds. 
Next, let $V(G) > 1$, and assume that the the statement holds for every critical quasicomb with a smaller number of vertices. 
Let $F$ be a minimum join of $(G, T; A, B)$. 
Let $(G', T', A', B')$ be a maximal subgraft of $(G, T; A, B)$ 
that is a member of $\critical{r}$  for which $E_G[B', A\setminus A'] \cap F = \emptyset$.

Suppose  $V(G') \subsetneq V(G)$. 
Lemma~\ref{lem:ear2critical} implies that $(G', T'; A', B')$ is a critical quasicomb with root $r$. 
Therefore, Lemma~\ref{lem:critical2increment} implies that 
$(G, T; A, B)$ has an ear graft $(P, T''; A'', B'')$ relative to $(G', T'; A', B')$ 
such that $V(P)\setminus V(G') \neq \emptyset$ and $E_G[B'', A\setminus (A' \cup A'')]\cap F = \emptyset$.   
Then, $(G', T'; A', B') \oplus (P, T''; A'', B'')$ is a member of $\critical{r}$ for which $E_G[B'\cup B'', A\setminus (A' \cup A'')]\cap F = \emptyset$.  
This contradicts the maximality of $(G', T'; A', B')$. 
Hence, we obtain $V(G') = V(G)$. 
This clearly implies $(G, T; A, B) \in \critical{r}$.  
The lemma is proved. 
\end{proof}

Combining Lemmas~\ref{lem:ear2critical} and \ref{lem:critical2ear}, 
we now obtain Theorem~\ref{thm:critical2char}. 

\begin{theorem} \label{thm:critical2char} 
A bipartite graft with an ordered bipartition  is a critical quasicomb with root $r$ 
if and only if it is a member of $\critical{r}$. 
\end{theorem}

\begin{definition}
Let $(G, T; A, B) \in \critical{r}$. 
Let $\{ (G_i, T_i; A_i, B_i) \in \critical{r} : i = 1, \ldots, l\}$, where $l \ge 1$, be a family of grafts such that  
\begin{rmenum} 
\item $(G_1, T_1; A_1, B_1) = ( ( \{r\}, \emptyset), \emptyset; \emptyset, \{r\} )$, 
\item for each $i \in \{1, \ldots, l\} \setminus \{1\}$, 
$(G_{i+1}, T_{i+1}; A_{i+1}, B_{i+1}) = (G_i, T_i; A_i, B_i) \oplus (P_i, T_i'; A_i', B_i')$, 
where $(P_i, T_i'; A_i', B_i')$ is an effective ear graft relative to $(G_i, T_i; A_i, B_i)$, and 
\item $(G_l, T_l; A_l, B_l) = (G, T; A, B)$. 
\end{rmenum} 
We call the family $\{ (P_i, T_i'; A_i', B_i'): i= 1,\ldots, l\}$ a {\em graft ear decomposition} of $(G, T; A, B)$. 
\end{definition} 

\begin{definition} 
Let $(G, T; A, B)$ be a critical quasicomb with root $r\in B$, and let $F \subseteq E(G)$. 
Let  $\mathcal{P}$ be a family of bipartite grafts with 
 an  ordered bipartition that is a graft ear decomposition of $(G, T; A, B)$. 
We say that $\mathcal{P}$ is {\em $F$-balanced} if
 $F\cap E(P)$ is a minimum join for each $(P, T'; A', B') \in \mathcal{P}$. 
\end{definition}  

The next statement clearly follows from the proof of Lemma~\ref{lem:critical2ear}. 

\begin{theorem} \label{thm:ear2balanced}  
Let $(G, T; A, B)$ be a critical quasicomb with root $r\in B$ and $F$ be a minimum join.   
Then, there is a graft ear decomposition of $(G, T; A, B)$ that is $F$-balanced. 
\end{theorem}

\section{Bipartite Cathedral Order in Combs}  \label{sec:cath}
\subsection{Critical Sets} \label{sec:cath:cr} 

In Section~\ref{sec:cath}, we prove our new decomposition. 
In Section~\ref{sec:cath:cr}, we define the concept of critical sets, 
and use this to introduce a binary relation between factor-components in a comb. 
We prove that this binary relation is a partial order in later sections. 

\begin{definition} 
Let $(G, T; A, B)$ be a comb. 
Let $G_0 \in \tcomp{G}{T}$. 
We say that  $X\subseteq V(G)$ is a {\em critical} set for $G_0$ 
if $X$ is separating, $V(G_0)\subseteq X$ holds, and 
$(G, T)[X]/G_0$ is a critical quasicomb with root $g_0$  
for which $(A\cap X) \setminus V(G_0)$ and  $(B\cap X)\setminus V(G_0) \cup \{g_0\}$ are the spine and tooth sets, respectively,  
where $g_0$ is the contracted vertex corresponding to $G_0$. 
\end{definition}

\begin{definition} 
Let $(G, T; A, B)$ be a comb. 
Let $G_1, G_2 \in \tcomp{G}{T}$. 
We say that $G_1 \preceq G_2$  if there is a critical set for $G_1$ with $V(G_2)\subseteq X$. 
Furthermore, we say that  $X \subseteq V(G)$ is a critical set for $G_1 \preceq G_2$  
if $X$ is a critical set for $G_1$ with $V(G_2) \subseteq X$. 
\end{definition}

\begin{remark} 
Note that a critical set is determined on the premise of the given comb. 
Hence, a critical set for a comb $(G, T; A, B)$ 
may not be a critical set for $(G, T; B, A)$ 
even if $(G, T; B, A)$ is also a comb. 
\end{remark} 

\begin{remark} 
If $X$ is a critical set for $G_0$, then  $F\cap E[X\setminus V(G_1)]$ is clearly a minimum join of the quasicomb $(G, T)[X]/G_1$. 
\end{remark}

The next lemma is a fundamental characterization of critical sets that is analogical to Lemma~\ref{lem:critical2dist} for critical quasicombs.

\begin{lemma} \label{lem:crset2path} 
Let $(G, T; A, B)$ be a comb, and let $F$ be a minimum join. 
Let $G_0\in \tcomp{G}{T}$, and let $X\subseteq V(G)$ be a separating set with $V(G_0)\subseteq X$. 
Then, the following two are equivalent. 	
\begin{rmenum} 
\item \label{item:crset} $X$ is a critical set for $G_0$.  
\item \label{item:path} For every $x\in (X\cap A)\setminus V(G_0)$, there is a path of $F$-weight $1$ between $x$ and a vertex $y\in V(G_0)$ whose vertices except for $y$ are contained in $X\setminus V(G_0)$. 
For every $x\in (X\cap B)\setminus V(G_0)$, there is a path of $F$-weight $0$ between $x$ and a vertex $y\in V(G_0)$ whose vertices except for $y$ are contained in $X\setminus V(G_0)$. 
\end{rmenum} 
\end{lemma} 
\begin{proof} 
First, we prove that \ref{item:crset} implies \ref{item:path}. 
If \ref{item:crset} holds, then the set $F\cap E[X\setminus V(G_0)]$ is clearly a minimum join of the critical quasicomb $(G, T)[X]/G_0$. 
Hence, \ref{item:path}  easily follows from Lemma~\ref{lem:critical2dist}. 

We next  assume \ref{item:path} and prove \ref{item:crset}. 

\begin{pclaim} \label{claim:bipartite} 
$E_G[A\cap V(G_0), B\cap (X\setminus V(G_0))] = \emptyset$.  
\end{pclaim} 
\begin{proof} 
Suppose, to the contrary, that there is an edge $e$ between $u\in A \cap V(G_0)$ and $v\in B \setminus V(G_0)$.  
Let $P$ be a path between $v$ and a vertex $w\in V(G_0)$ with $w_F(P) = 0$ and $V(P)\setminus \{w\} \subseteq X\setminus V(G_0)$.   
Under Lemma~\ref{lem:elemcomb2nonposi}, let $Q$ be a path of $G_0$ between $u$ and $w$ with $w_F(Q) = -1$. 
Then, $P + Q + e$ is a circuit of $F$-weight zero that contains non-allowed edge $e$, 
which contradicts Lemma~\ref{lem:circuit}.  
Thus, the claim follows.   
\end{proof} 

Therefore, $(G, T)[X]/G_0$ is a bipartite graft with color classes $( A \cap  X)  \setminus V(G_0) $ 
and $( B\cap X) \setminus V(G_0) \cup \{g_0\}$, where $g_0$ is the contracted vertex corresponding to $G_0$.  

\begin{pclaim} \label{claim:minjoin} 
$F\cap E[X\setminus V(G_0)]$ is a minimum join of $(G, T)[X]/G_0$. 
\end{pclaim} 
\begin{proof} 
Under Lemma~\ref{lem:minimumjoin}, 
it suffices to prove that $G$ has no path with negative $F$-weight between any two vertices in $V(G_0)\cap B$ 
whose vertices except for the ends are disjoint from $G_0$. 
This can easily be confirmed from Claim~\ref{claim:bipartite} and Lemma~\ref{lem:qcomb2path}. 
\end{proof} 

If follows from Claim~\ref{claim:minjoin}  that 
the bipartite graft $(G, T)[X]/G_0$ is a quasicomb with spine set $(A \cap  X) \setminus V(G_0) $ and 
tooth set $( B\cap X) \setminus V(G_0) \cup \{g_0\}$. 
According to Lemma~\ref{lem:qcomb2path}, in this quasicomb, 
the $F$-distance between $g_0$ and any spine or tooth vertex is clearly no less than $1$ or $0$, respectively; 
the assumption \ref{item:path} implies that it is exactly $1$ or $0$. 
Hence, \ref{item:crset} now follows from Lemma~\ref{lem:critical2dist}.  
This completes the proof. 
\end{proof}

\begin{remark} 
Note that those paths with $F$-weight $1$ or $0$ are always $F$-balanced. 
\end{remark}

\subsection{Effectively Balanced Ears}

\begin{definition} 
Let $(G, T; A, B)$ be a quasicomb, and let $F$ be a minimum join. 
An $F$-balanced ear $P$ relative to a set of vertices is {\em effectively $F$-balanced}  
if it is round or if it is straight and $F\cap E(P)$ is a perfect matching of $P - \{s\}\cap A - \{t\}\cap B$, 
where $s$ and $t$ are the ends of $P$ between which $t$ is the bond. 
\end{definition}

The next lemma is the graph analogue of Lemma~\ref{lem:ear2path} and can  immediately be confirmed. 

\begin{lemma} \label{lem:graphear2path} 
Let $(G, T; A, B)$ be a quasicomb, and let $F$ be a minimum join. 
Let $P$ be an effectively $F$-balanced ear relative to a set of vertices, 
and let $s$ and $t$ be the bonds of $P$. 
Then, the following properties hold. 
\begin{rmenum} 
\item For each $x\in V(P)\cap A$, there exists either $r\in \{s, t\} \cap A$  
with $w_F(xPr) = 0$  or $r\in \{s, t\} \cap B$ with $w_F(xPr) = 1$. 
\item For each $x\in V(P)\cap B$, there exists either $r\in \{s, t\} \cap A$  
with $w_F(xPr) = -1$  or $r\in \{s, t\} \cap B$ with $w_F(xPr) = 0$.  
\end{rmenum} 
\end{lemma}

Let $(G, T; A, B)$ be a comb, $F$ be a minimum join, and  $X\subseteq V(G)$ be a critical set  for $G_0 \in \tcomp{G}{T}$. 
Under Theorem~\ref{thm:critical2char}, 
we call an ear decomposition $\mathcal{P}$ of $G[X]/G_0$  an {\em ear decomposition}  of the critical set $X$.  
Furthermore, $\mathcal{P}$ is said to be {\em $F$-balanced} if every ear is effectively $F$-balanced. 
The next lemma is easily confirmed from Theorem~\ref{thm:ear2balanced}. 

\begin{lemma} \label{lem:crset2ear} 
Let $(G, T; A, B)$ be a comb, and let $F\subseteq E(G)$ be a minimum join. 
Let $G_0 \in \tcomp{G}{T}$, and let $X\subseteq V(G)$ be a critical set for $G_0$. 
Then, there exists an $F$-balanced ear decomposition of $X$.   
\end{lemma}

\subsection{Guides} 

In this section, we introduce the concept of guides 
and lemmas regarding them to be used for proving main results.

\begin{definition} 
Let $(G, T; A, B)$ be a comb, and let $F$ be a minimum join.  
Let $X_0 \subseteq V(G)$. 
We call a pair $\guide{X}{P}{F}$  associated with $F$   a {\em guide} for $X_0$ if   
\begin{rmenum} 
\item $X$ is a subset of $V(G) \setminus X_0 $, 
\item $P$ is an $F$-balanced ear relative to $X_0$ such that the bonds are in $B\cap X_0$ and the necks are not in $F$, 
\item 
$\internal{P} \subseteq X$ holds, and,  
\item  
for every $x\in X$, there exists an $F$-balanced path $Q$ between $x$ and a bond $y$ of $P$ such that 
$V(Q) \setminus \{y\} \subseteq X$ and $w_F(Q) \in \{0 ,1\}$  hold. 
\end{rmenum} 
Such path $Q$ with these properties is said to be {\em along} the guide $\guide{X}{P}{F}$. 
\end{definition}

\begin{remark} 
Every  path $Q$ that is along the guide $\guide{X}{P}{F}$ contains the neck of $P$. 
Because the bonds of $P$ are in $B\cap X_0$, 
$w_F(Q) \in \{0 ,1\}$ implies that $w_F(Q) = 0$ for $x\in B$, whereas $w_F(Q) = 1$ for $x\in A$.  
  Lemma~\ref{lem:qcomb2path} further implies that $Q$ is $F$-balanced, and, 
if $x\in B$ holds, then the edge of $Q$ connected to $x$ is in $F$. 
\end{remark}

We now present two fundamental observations regarding guides. 
The next lemma is easily confirmed from the definition of guides and Lemma~\ref{lem:qcomb2path}.

\begin{lemma} \label{lem:closed} 
Let $(G, T; A, B)$ be a comb, and let $F$ be a minimum join.  
Let $X_0 \subseteq V(G)$,  and $\guide{X}{P}{F}$ be a guide for $X_0$. 
Then, $\parcut{G}{(X_0 \cup X) \cap B} \cap F = \emptyset$. 
\end{lemma}

The next lemma can easily be confirmed by considering the concatenation of $F$-balanced paths under Lemma~\ref{lem:graphear2path}.

\begin{lemma} \label{lem:guide2increment} 
Let $(G, T; A, B)$ be a comb, and let $F$ be a minimum join. 
Let $G_0 \in \tcomp{G}{T}$, and let $\guide{X}{P}{F}$ be a guide for $G_0$. 
Let $Q$ be an effectively $F$-balanced ear relative to  $X$ with $V(Q)\cap V(G_0) = \emptyset$.  
Then, $\guide{X\cup V(Q)}{P}{F}$ is a guide for $G_0$. 
\end{lemma}

In the following, we provide two more lemmas on guides, Lemmas~\ref{lem:guide2ear} and \ref{lem:guide2comp}. 
The next lemma describes the relationship between guides and ear decompositions of  critical sets, 
and can be proved using Lemma~\ref{lem:guide2increment}.

\begin{lemma} \label{lem:guide2ear} 
Let $(G, T; A, B)$ be a comb, and let $F$ be a minimum join. 
Let $G_0 \in \tcomp{G}{T}$, and let $X$ be a critical set for $G_0$. 
Let $\mathcal{P}$ be an $F$-balanced ear decomposition of $X$. 
Then, for each $P \in \mathcal{P}$, 
there is a guide $\guide{Y}{Q}{F}$ for $G_0$ 
 with $\internal{P}\subseteq Y$, $V(P)\subseteq Y \cup V(G_0)$, and $Y \subseteq \bigcup \{ V(P'): P' \in \mathcal{P} \mbox{ with } P'\eo P \}$. 
\end{lemma} 
\begin{proof}  
We prove this lemma by induction on the height of $P \in \mathcal{P}$ regarding $\eo$. 
First, if the height of $P$ is equal to one, 
that is, $P$ is an ear relative to $G_0$, 
then $\guide{\internal{P}}{P}{F}$ is a desired guide.  
Next, consider the case where the height of $P$ is more than one. 
Assume that the statement holds for every member of $\mathcal{P}$ whose height is less.  
Let $P_1$ and $P_2$ be immediate lower elements of $P$ in the poset $(\mathcal{P}, \eo)$, 
which are possibly identical. 
Under the induction hypothesis,  for each $i \in \{1, 2\}$, 
let $\guide{X_i}{Q_i}{F}$ be a guide for $G_0$ 
 with $\internal{P_i} \subseteq X_i$, $V(P_i) \subseteq X_i\cup V(G_0)$, and $X_i \subseteq \bigcup \{V(P'):  P' \in \mathcal{P} \mbox{ with } P'\eo P_i \}$. 
Note $\internal{P} \cap X_i = \emptyset$ for each $i \in \{1, 2\}$.

First, we consider the case where either $X_1$ or $X_2$ contains every bond of $P$; 
without loss of generality, assume that $X_1$ does. 
Then, Lemma~\ref{lem:guide2increment} implies that  $\guide{X_1 \cup V(P)}{Q_1}{F}$ is a desired guide.

Next, assume that $P$ has two bonds $s_1$ and $s_2$ 
that are contained in $X_1\setminus X_2$ and $X_2\setminus X_1$, respectively. 
For each $i\in \{1, 2\}$,  
let $R_i$ be an $F$-balanced path between $s_i$ and a bond $t_i$ of $Q_i$ 
that is taken along the guide $\guide{X_i}{Q_i}{F}$. 
If $R_1$ and $R_2$ are disjoint except for their ends in $G_0$, 
then $R_1 + P + R_2$ is an $F$-balanced ear relative to $G_0$. 
Hence,  $\guide{\internal{ R_1 + P + R_2}}{R_1 + P + R_2}{F}$ is a desired guide in this case.  
Hence, assume now that $R_1$ and $R_2$ share vertices in $X_1 \cap X_2$. 
Trace $R_2$ from $s_2$, and let $t_2'$ be the first encountered vertex in $X_1$.  
Then, $P + s_2R_2t_2'$ is an $F$-balanced ear relative to $X_1$, which is round and, accordingly, effective. 
Lemma~\ref{lem:guide2increment} implies that $\guide{X_1 \cup V(P) \cup V(s_2R_2t_2')}{Q_1}{F}$ is a desired guide for $G_0$. 
This completes the proof of the lemma. 
\end{proof}

Lemma~\ref{lem:guide2ear} implies the next lemma, which will be used in proving the antisymmetry of $\preceq$ in Lemma~\ref{lem:antisymmetry}.

\begin{lemma} \label{lem:guide2comp} 
Let $(G, T; A, B)$ be a comb, and let $F$ be a minimum join. 
Let $G_0 \in \tcomp{G}{T}$, and let $H\in \tcomp{G}{T} \setminus \{ G_0 \}$ be a factor-component with $G_0\preceq H$. 
Then, there is a guide $\guide{X}{P}{F}$ with $V(H)\subseteq X$. 
\end{lemma} 
\begin{proof} 
Let $Y\subseteq V(G)$ be a critical set for $G_0\preceq H$. 
Let $\mathcal{P}=\{P_1, \ldots, P_k\}$, where $k\ge 1$, be an $F$-balanced ear decomposition of the critical set $Y$.  
Let $i \in \{1, \ldots, k\}$ be an index  with $P_i\cap V(H)\neq \emptyset$. 
Lemma~\ref{lem:guide2ear} implies that there is a guide $\guide{X}{P}{F}$ for $G_0$ with $V(P_i)\subseteq X$. 
This implies $X\cap V(H) \neq \emptyset$.  

\begin{pclaim} \label{claim:a} 
$V(H)\cap X\cap A \neq \emptyset$. 
\end{pclaim} 
\begin{proof} 
Let $x\in V(H)\cap X$. If $x\in A$ holds, then the claim obviously holds. Hence, assume $x\in B$. 
Let $e$ be the edge from $F$ that is connected to $x$. 
Lemma~\ref{lem:closed} implies that  $e$ joins $x$ and a vertex $w$ in $X\cap A$. 
Because $e\in E(H)$  holds, we have $w \in V(H)\cap X \cap A$. 
The claim is proved. 
\end{proof}

\begin{pclaim} \label{claim:extend} 
$\guide{X\cup V(H)}{P}{F}$ is a guide for $G_0$. 
\end{pclaim} 
\begin{proof} 
This claim obviously follows if we prove, for every $x\in V(H)\setminus X$, that 
there is a path between $x$ and a bond of $P$ with $F$-weight  $0$ or $1$  whose vertices except for the bond are contained in $X$. 
Let $x\in V(H)\setminus X$.  
Under Claim~\ref{claim:a},  let $y\in X\cap V(H) \cap A$. 
Under Lemma~\ref{lem:elemcomb2nonposi}, let $Q$ be an $F$-shortest path in $H$ between $x$ and $y$.  
Lemma~\ref{lem:qcomb2path} implies that the edge of $Q$ connected to $x$ is in $F$. 
Trace $Q$ from $x$, and let $z$ be the first encountered vertex in $X$. 
Then, Lemmas~\ref{lem:qcomb2path} and \ref{lem:closed} further imply that $w_F(xQz) = -1$ if $z\in A$ holds, whereas $w_F(xQz) = 0$ if $x\in A$ holds. 
Let $R$ be a path between $x$ and a vertex in $V(G_0)$ that is taken along the guide $\guide{X}{P}{F}$. 
Then, $xQz + R$ is a desired path that proves the claim. 
\end{proof} 

Thus, Claim~\ref{claim:extend} proves the lemma. 
\end{proof}

\subsection{Antisymmetry}

In this section, we prove the antisymmetry of $\preceq$ using Lemma~\ref{lem:guide2comp}.

\begin{lemma} \label{lem:antisymmetry} 
Let $(G, T; A, B)$ be a comb. 
Let $G_1, G_2 \in \tcomp{G}{T}$ be factor-components with $G_1 \preceq G_2$ and $G_2 \preceq G_1$. 
Then, $G_1$ and  $G_2$ are identical. 
\end{lemma} 
\begin{proof} 
Suppose $G_1 \neq G_2$. 
Let $F$ be a minimum join of $(G, T)$. 
According to Lemma~\ref{lem:guide2comp}, 
there is a guide $\guide{X}{Q}{F}$ for $G_1$ with $V(G_2)\subseteq X$. 
Let $s$ and $t$  be the possibly identical bonds of $Q$.  
Note $s, t\in V(G_1)\cap B$.  
\begin{pclaim} \label{claim:stpath} 
If a path between $s$ and $t$ is of $F$-weight  $-2$ or less, then it has a vertex in $\internal{Q}$. 
\end{pclaim} 
\begin{proof} 
Suppose that $Q'$ is a path between $s$ and $t$ with $V(Q')\cap \internal{Q} =  \emptyset$ and $w_F(Q') \le -2$. 
Then, $Q + Q'$ is a circuit of $F$-weight zero or less that contains edges from $\parcut{G}{G_1}$. 
This contradicts Lemmas~\ref{lem:minimumjoin} or \ref{lem:circuit}. 
The claim is proved. 
\end{proof} 

Let $R$ be a path such that the number of edges between $s$ and the midvertex 
takes the minimum value among all $F$-shortest paths in $G_1$ between $s$ and $t$. 
Let $r$ be the midvertex of $R$.

Because $G_2 \preceq G_1$ and $r \in V(G_1)\cap B$ hold, 
there is an $F$-balanced path $P$ with $w_F(P) = 0$ between $r$ and a vertex $u \in V(G_2)$ 
in which the edge connected to $r$ is in $F$ and the vertices except for $u$ are disjoint from $G_1$.  
Trace $P$ from $r$, and let $x$ be the first encountered vertex in $V(sRr - r) \cup X$.   
Trace $R$ from $t$, and let $z$ be the first encountered vertex in $rPx$.   

\begin{pclaim} \label{claim:trz} 
The vertex $z$ is in $B$, and  is connected to an edge from $F$ in the path $zPx$. 
Accordingly,   $tRz + zPx$ is an $F$-balanced path between $t$ and $x$ in which $t$ is connected to an edge from $F$. 
\end{pclaim} 
\begin{proof} 
Suppose the contrary. 
Then, $tRz + zPr + rRs$ is an $F$-balanced path between $s$ and $t$ in which both ends are connected to edges from $F$. 
Lemma~\ref{lem:qcomb2path} implies that its $F$-weight is $-2$. 
This contradicts Claim~\ref{claim:stpath}. 
The claim is proved. 
\end{proof}

\begin{pclaim} \label{claim:path2connect} 
$x\in X$ holds. 
\end{pclaim} 
\begin{proof} 
Suppose $x\in V(sRr-r)$.  Then, $s \neq r$ clearly holds.  
Either $rPx + xRr$ or $rPx + xRs$ is $F$-balanced. 
First, assume that $rPx + xRr$ is $F$-balanced. 
Then, $rPx + xRr$ is a circuit of $F$-weight zero. 
Therefore, Lemma~\ref{lem:circuit} implies that $rPx$ is a path in $G_1$. 
Accordingly, Claim~\ref{claim:trz} implies that 
$sRx + xPz + zRt$ is a path in $G_1$ between $s$ and $t$ whose $F$-weight is zero, 
and $x$ is the midvertex of this path. 
However, because $|E(sRx)| < |E(sRr)|$ holds, this contradicts the definition of $R$. 

Therefore, $rPx + xRs$ is $F$-balanced. 
Then, Lemma~\ref{lem:qcomb2path} implies that $tRz + zPx + xRs$ is a path of $F$-weight $-2$ between $s$ and $t$. 
This contradicts Claim~\ref{claim:trz}. 
This proves the claim. 
\end{proof}

Claim~\ref{claim:path2connect} implies that $sRr + P$ is an $F$-balanced path of weight zero between $s$ and $u$ 
in which $s$ is connected to an edge from $F$. 
Let $L$ be an $F$-balanced path between $x$ and an either bond of $Q$ that is taken along $\guide{X}{Q}{F}$. 
Note $sRr + rPx + L$ is $F$-balanced.  
Trace $L$ from $x$, and let $y$ be the first encountered vertex in $P$.     
Either one of the following  holds: 
\begin{rmenum} 
\item \label{item:a} $y\in A$, 
\item \label{item:bs} $y\in B$ and $w_F(yQs) = 0$, or 
\item \label{item:bt} $y\in B$ and $w_F(yQt) = 0$. 
\end{rmenum} 

For the cases \ref{item:a} and \ref{item:bs},  $sRr + rPx + xLy + yQs$ is  $F$-balanced. 
Consequently, it is a circuit of $F$-weight zero that contains non-allowed edges from $\parcut{G}{G_1}$, 
which contradicts Lemma~\ref{lem:circuit}. 
Hence, \ref{item:bt} follows. 
Thus, Claim~\ref{claim:trz} implies that $tRz + zPx + xLy + yQt$ is $F$-balanced and is accordingly a circuit of $F$-weight zero that contains non-allowed edges from $\parcut{G}{G_1}$, 
which contradicts Lemma~\ref{lem:circuit}.  

Thus, $G_1 = G_2$ is obtained. 
This completes the proof. 
\end{proof}

\subsection{Transitivity} 

We now prove the transitivity of $\preceq$ using Lemmas~\ref{lem:crset2path} and \ref{lem:antisymmetry}. 

\begin{lemma} \label{lem:transitivity} 
Let $(G, T; A, B)$ be a comb. 
Let $G_1, G_2, G_3\in \tcomp{G}{T}$. 
Let $X, Y \subseteq V(G)$ be critical sets for $G_1 \preceq G_2$ and $G_2\preceq G_3$, respectively. 
Then, $X\cup Y$ is a critical set for $G_1 \preceq G_3$. 
\end{lemma} 
\begin{proof} 
Let $F$ be a minimum join of $(G, T)$. 
Under Lemma~\ref{lem:crset2path}, it suffices to prove that,  
for every $x\in X\cup Y$,  
 if $x\in A$ (resp. $x\in B$) holds, 
 then there is a path of $F$-weight $1$ (resp. $0$) 
 between $x$ and a vertex in $G_1$ whose vertices except for the end in $G_1$ are contained in $(X\cup Y) \setminus V(G_1)$. 

Because this obviously holds for $x\in X$, we only need to prove it for $x\in Y\setminus X$. 
If $G_1$ and $G_2$ are identical, then the claim obviously holds. 
Hence, assume that $G_1$ and $G_2$ are distinct. 
Then, Lemma~\ref{lem:antisymmetry} implies  $V(G_1) \cap Y = \emptyset$.

Let $x\in (Y\setminus X)\cap B$. 
Because $Y$ is a critical set for $G_2$, 
Lemma~\ref{lem:crset2path} implies that 
there is an $F$-balanced path $P$ with $w_F(P) = 0$ between $x$ and a vertex in $G_2$ such that $V(P)\subseteq Y$ holds 
and the edge connected to $x$ is in $F$.  
Trace $P$ from $x$, and let $y$ be the first encountered vertex in $X$. 
Note  $y \in X \setminus V(G_1)$.    

Additionally, because $X$ is a critical set for $G_1$, 
there is an $F$-balanced path $Q$ between $y$ and a vertex $z \in V(G_1)$  with $V(Q) \setminus \{z\} \subseteq X \setminus V(G_1)$. 
In $Q$, the end $z$ is connected to an edge from $E(G)\setminus F$; 
furthermore, if $x\in B$ holds, the end $y$ is connected to an edge from $F$. 
Thus, $P + Q$ is an $F$-balanced path between $x$ and $z$ with $V(P + Q) \setminus \{z\} \subseteq ( X\cup Y ) \setminus V(G_1)$ 
whose edges connected to $x$ and $z$ are in $F$ and $E(G)\setminus F$, respectively. 
Lemma~\ref{lem:qcomb2path} implies  $w_F(P+Q) = 0$. 
The case with $x\in A$ can be proved by a similar discussion. 
This proves the lemma.

\end{proof}

\subsection{Proof of the Partial Order} \label{sec:cath:order} 

Lemmas~\ref{lem:antisymmetry} and \ref{lem:transitivity} now prove that $\preceq$ is a partial order. 

\begin{theorem} \label{thm:order} 
Let $(G, T; A, B)$ be a comb. 
Then, $\preceq$ is a partial order over $\tcomp{G}{T}$. 
\end{theorem} 
\begin{proof} 
Reflexivity of $\preceq$ is obvious. 
Lemmas~\ref{lem:antisymmetry} and \ref{lem:transitivity} immediately imply antisymmetry and transitivity, respectively.  
Thus, $\preceq$ is a partial order. 
\end{proof}

Theorem~\ref{thm:order} states that $(\tcomp{G}{T}, \preceq)$ is a poset for every comb $(G, T; A, B)$. 
We call this poset the {\em bipartite cathedral decomposition} for combs.

\section{Cathedral-Type Relationship between Upper Bounds and Equivalence Classes}  \label{sec:relationship}

Even though Theorem~\ref{thm:order} is obtained without assuming Theorem~\ref{thm:tkl}, 
there is a canonical relationship between the general Kotzig-Lov\'asz  
and  bipartite cathedral decompositions.  
In this section, we prove this relationship.

\begin{lemma}[Kita~\cite{kita2020bipartite}] \label{lem:earend2sim} 
Let $(G, T; A, B)$ be a comb, and let $F$ be a minimum join. 
Let $H\in\tcomp{G}{T}$. 
Let $P$ be an $F$-balanced round ear relative to $H$ with bonds $s$ and $t$. 
Then, $s, t\in B\cap V(H)$ holds, and $\distgtf{G}{T}{F}{s}{t} = 0$. 
\end{lemma}

Lemmas~\ref{lem:guide2comp} and \ref{lem:earend2sim} derive Theorem~\ref{thm:upper}. 

\begin{theorem} \label{thm:upper} 
Let $(G, T; A, B)$ be a comb, and let  $G_0 \in\tcomp{G}{T}$. 
Let $K$ be a connected component of $G[\vup{G_0}]$. 
Then, there exists $S\in \pargtpart{G}{T}{G_0}$ with $S\subseteq B$ 
such that $\parNei{G}{K} \cap V(G_0) \subseteq S$ holds. 
\end{theorem} 
\begin{proof} 
Let $F$ be a minimum join of $(G, T; A, B)$.  
According to Lemma~\ref{lem:guide2comp}, 
for each $v\in V(K)$, there is a guide $\guide{X_v}{Q_v}{F}$ for $G_0$ with $v\in X_v$. 
Hence,  
there is a family $\mathcal{T}$ of guides for $G_0$ such that $V(K) \subseteq \bigcup \{ X: \guide{X}{Q}{F} \in \mathcal{T} \}$. 
For each $S\in \pargtpart{G}{T}{G_0}$ with $S\subseteq B$, 
we define the subfamily $\mathcal{T}_S$ of $\mathcal{T}$ as follows: 
Let $\guide{X}{Q}{F}  \in \mathcal{T}_S$ if the bonds of $Q$ are in $S$. 
Lemma~\ref{lem:earend2sim} implies that $\{ \mathcal{T}_S : S \in \pargtpart{G}{T}{G_0}  {\mbox  {\rm  ~with~}} S\subseteq B \}$ 
is a partition of $\mathcal{T}$. 

\begin{pclaim} \label{claim:disjoint} 
Let $S$ and $T$ be distinct members from $\pargtpart{G}{T}{G_0}$ with $S, T\subseteq B$. 
Then, $X_S \cap X_T = \emptyset$ for every $\guide{X_S}{Q_S}{F} \in\mathcal{T}_S$ and every $\guide{X_T}{Q_T}{F} \in\mathcal{T}_T$. 
\end{pclaim} 
\begin{proof} 
Suppose that the claim fails, and let $x \in X_S \cap X_T$. 
Let $P$ be a path between $x$ and a bond $t$ of $Q_T$ that is taken along the guide $\guide{X_T}{Q_T}{F}$.  
Trace $P$ from $t$, and let $x'$ be the first encountered vertex in $X_S$. 
Let $P'$ be a path between $x'$ and a bond $s$ of $Q_S$ that is taken along the guide $\guide{X_S}{Q_S}{F}$. 
Then, $P + P'$ is an $F$-balanced ear relative to $G_0$ whose bonds are $s\in S$ and $t\in T$. 
This contradicts Lemma~\ref{lem:earend2sim}. 
Thus, the claim is proved. 
\end{proof} 

Under Claim~\ref{claim:disjoint}, the next claim follows. 
\begin{pclaim} \label{claim:noadjacent} 
Let $S$ and $T$ be distinct members from $\pargtpart{G}{T}{G_0}$ with $S, T\subseteq B$, 
and let $\guide{X_S}{Q_S}{F} \in\mathcal{T}_S$ and $\guide{X_T}{Q_T}{F} \in\mathcal{T}_T$. 
Then, $E_G[X_S, X_T] = \emptyset$. 
\end{pclaim} 
\begin{proof} 
Suppose $E_G[X_S, X_T] \neq \emptyset$, 
and let $u_S \in X_S$ and $u_T \in X_T$ be vertices with $u_Su_T\in E(G)$.   
For each $\alpha \in \{S, T\}$, 
let $P_\alpha$ be a path between $u_\alpha$ and a bond $r_\alpha$ of $Q_\alpha$ that is taken along the guide $\guide{X_\alpha}{Q_\alpha}{F}$. 
According to Claim~\ref{claim:disjoint}, $P_S$ and $P_T$ are disjoint. 
Hence, $P_S + u_Su_T + P_T$ is an $F$-balanced ear relative to $G_0$ whose bonds are $r_S \in S$ and $r_T \in T$. 
This contradicts Lemma~\ref{lem:earend2sim}. 
Hence, the claim is proved. 
\end{proof} 

From Claims~\ref{claim:disjoint} and \ref{claim:noadjacent}, the next claim follows. 
\begin{pclaim} \label{claim:neighbor} 
There exists $S \in \pargtpart{G}{T}{G_0}$ with $S\subseteq B$ such that $V(K) = \bigcup \{ X : \guide{X}{Q}{F} \in \mathcal{T}_S\}$. 
\end{pclaim}

We now prove the lemma by proving  $\parNei{G}{K}\cap V(G_0) \subseteq S$.  
Contrary to the statement, suppose that there exists $v\in \parNei{G}{K} \cap V(G_0)$ with $v\not\in S$. 
Let $u\in V(K)$ be a vertex with $uv\in E(G)$. 
Under Claim~\ref{claim:neighbor},  let $\guide{X}{Q}{F} \in \mathcal{T}_S$ be a guide for $G_0$ with $u\in X$. 
Let $P$ be an $F$-balanced path between $u$ and a bond $r$ of $Q$ that is taken along the guide $\guide{X}{Q}{F}$. 
Then, $P + uv$ is an $F$-balanced ear relative to $G_0$ whose bonds are $r$ and $v$. 
This contradicts Lemma~\ref{lem:earend2sim}. 
The theorem is proved. 
\end{proof}

\begin{ac} 
This study is supported by JSPS KAKENHI Grant Number 18K13451. 
\end{ac}

\bibliographystyle{splncs03.bst}
\bibliography{tbicath.bib}

\end{document}